\documentclass[11pt]{article}

\usepackage[margin=1in]{geometry}
\usepackage{amsmath,amssymb,amsthm,mathtools}
\usepackage{enumitem}
\usepackage{hyperref}
\usepackage{microtype}
\usepackage{authblk}
\hypersetup{colorlinks=true,linkcolor=blue,citecolor=blue,urlcolor=blue}

\numberwithin{equation}{section}
\setlist{nosep}

\newtheorem{theorem}{Theorem}[section]
\newtheorem{lemma}[theorem]{Lemma}
\newtheorem{proposition}[theorem]{Proposition}
\newtheorem{corollary}[theorem]{Corollary}
\newtheorem{conjecture}[theorem]{Conjecture}
\theoremstyle{remark}
\newtheorem{remark}[theorem]{\textbf{Remark}}

\DeclareMathOperator{\co}{co}
\newcommand{\cA}{\mathcal A}
\newcommand{\cB}{\mathcal B}
\newcommand{\cF}{\mathcal F}
\newcommand{\cG}{\mathcal G}
\newcommand{\cS}{\mathcal S}
\newcommand{\ind}{\mathbf 1}
\newcommand{\qbinom}[2]{\genfrac{[}{]}{0pt}{}{#1}{#2}_{q}}
\def\q{\hfill\rule{1ex}{1ex}}
\title{Degree Power Sums in Extremal Set Systems}

\author[1]{Mengyu Cao\thanks{E-mail: \texttt{myucao@ruc.edu.cn}. Supported by the National Natural Science Foundation of China (12301431) and Beijing Natural Science Foundation (1262010).}}
\author[2]{Mei Lu\thanks{E-mail: \texttt{lumei@tsinghua.edu.cn}. M. Lu is supported by the National Natural Science Foundation of China (Grant 12571372) and Beijing Natural Science Foundation (Grant 1262010).}}
\author[2]{Haixiang Zhang\thanks{Corresponding author. E-mail: \texttt{zhang-hx22@mails.tsinghua.edu.cn}.}}

\affil[1]{\small Institute for Mathematical Sciences, Renmin University of China, Beijing 100086, China}
\affil[2]{\small Department of Mathematical Sciences, Tsinghua University, Beijing 100084, China}

\date{}

\begin{document}
\maketitle

\begin{abstract}
For a family $\cF\subseteq\binom{[n]}k$ and $R\in\binom{[n]}r$, let
$d_{\cF}(R)=|\{F\in\cF:R\subseteq F\}|$ and
$\ell_{r,p}(\cF)=\sum_{R\in\binom{[n]}r}d_{\cF}(R)^p$; write
$\co_p(\cF)=\ell_{k-1,p}(\cF)$ for the codegree power sum.
We introduce a method that uses convexity to extend sharp bounds for
degree sums and sums of squared degrees to real powers, while retaining
all equality cases. The method bounds $x^p$ by quadratic polynomials or
by a continuous function that is linear on each of two intervals.
These functions agree with $x^p$ at the degrees of the proposed
extremal family, so the argument requires no bounds for sums of
higher powers. For families with bounded matching number, we instead
use a bound for $\sum_R\max\{d_{\cF}(R)-s,0\}$ together with the
degree sum. We give three exact applications. First, a full $t$-star
maximizes $\co_p$ among $t$-intersecting families for every real
$p\geq2$ in the sharp range $n\geq(t+1)(k-t+1)$, with all equality
cases determined. This extends the quadratic theorem of Wu and Zhang
to real exponents and answers a problem of Zhou and Yuan throughout
the sharp Erd\H{o}s--Ko--Rado range. Second, among intersecting
families with $n\geq2k$, a full star maximizes $\ell_{r,p}$ for every
$1\leq r\leq k-1$ and real $p\geq2$, again with all equality cases
determined. Third, if $\nu(\cF)\leq s$ and $n\geq(2s+1)k-s$, then
for every real $p\geq1$, $\co_p(\cF)$ is uniquely maximized, up to
permutation, by all $k$-sets meeting a fixed $s$-set. This removes the
integrality restriction on $p$ and replaces previous cubic thresholds
or assumptions that $n$ is sufficiently large with an explicit linear
range valid for arbitrary $k$.

\end{abstract}


\section{Introduction}

Write $[n]=\{1,\ldots,n\}$, and let $\binom Xr$ denote the family
of all $r$-subsets of a finite set $X$. For integers $a\leq b$, write
$[a,b]=\{a,a+1,\ldots,b\}$. We use $\ind_{\mathcal A}$ for the
indicator function of a family $\mathcal A$, and $\ind_{\{Q\}}$ for
$1$ when a condition $Q$ holds and $0$ otherwise. We also write
$(y)_+=\max\{y,0\}$.

A family $\cF\subseteq\binom{[n]}k$ is \emph{$t$-intersecting} if
$|F\cap F'|\geq t$ for all $F,F'\in\cF$. A $1$-intersecting family
is called \emph{intersecting}. The classical
Erd\H{o}s--Ko--Rado theorem \cite{EKR} initiated the systematic study
of uniform intersecting families.  Katona's proof using cyclic permutations
\cite{KatonaSimple} became one of the standard tools of extremal set
theory, while algebraic proofs use spectral properties of inclusion matrices.  For broader accounts, see
\cite{GodsilMeagher,FranklTokushigeSurvey,FranklTokushigeBook}.  The
sharp $t$-intersecting form, proved by Wilson \cite{Wilson}, states
that
\begin{equation}\label{eq:classical-ekr}
  |\cF|\leq\binom{n-t}{k-t}
  \qquad\text{if}\qquad
  n\geq(t+1)(k-t+1).
\end{equation}
For $t<k$, the threshold is best possible.  The bound is attained by
a {\em full $t$-star}
\[
  \cS_T
  :=
  \left\{F\in\binom{[n]}k:T\subseteq F\right\},
  \qquad T\in\binom{[n]}t,
\]
and above the threshold this is the unique extremal family up to permutation.  At the threshold itself additional
equality families may occur.  The complete intersection theorem of
Ahlswede and Khachatrian \cite{AK} determined the maximum and all
extremal constructions for every $n,k,t$.  Filmus
\cite{Filmus} subsequently proved a weighted version for the biased
measure, showing that the extremal constructions from the complete
intersection theorem also arise for objectives other than cardinality.

Balogh, Clemen, and Lidick\'y
\cite{BCLSurvey,BCLTetrahedron} introduced a norm version of
hypergraph extremal problems based on degree vectors. Let $\cF\subseteq\binom{[n]}k$. For
$1\leq r\le k\leq n$ and $R\in\binom{[n]}r$, write
\[
  d_{\cF}(R):=
  |\{F\in\cF:R\subseteq F\}|
\]
and, for real $p>0$, define the {\em $r$-degree power sum}
\[
  \ell_{r,p}(\cF)
  :=
  \sum_{R\in\binom{[n]}r}d_{\cF}(R)^p.
\]
The range $1\leq r\leq k-1$ consists of the nontrivial degree
levels considered below.  When $r=1$, we abbreviate
$\ell_{1,p}(\cF)$ to
$\ell_p(\cF)$.  The highest nontrivial level $r=k-1$ is the
codegree level, and we write
\[
  \co_p(\cF)
  :=
  \ell_{k-1,p}(\cF).
\]
Thus $\ell_{k-1,p}(\cF)$ and $\co_p(\cF)$ denote the same
quantity.
For $p\geq1$, these power sums are the $p$-th powers of the
corresponding $\ell_p$-norms.  Taking the $p$-th root does not change
the extremal families, so we work throughout with the degree power sums.  For $1\leq r\leq k$,
at $p=1$,
\[
  \ell_{r,1}(\cF)=\binom kr|\cF|,
\]
and hence the problem is exactly the classical cardinality problem.
At the codegree level, $p=2$ gives
\[
  \co_2(\cF)
  =
  k|\cF|+
  2\left|\left\{\{F,F'\}\in\binom{\cF}{2}:
  |F\cap F'|=k-1\right\}\right|,
\]
so the quadratic objective records both the number of members of
$\cF$ and the number of unordered pairs $\{F,F'\}$ satisfying
$|F\cap F'|=k-1$.

Degree power sums also have a history outside intersection
theory.  Bey \cite{Bey} proved a sharp universal upper bound for the
sum of squares of the $r$-degrees in terms of the number of edges.
Related objectives weighting pairs by the size of their intersection
were investigated by Huang \cite{HuangIntersections}.  Gao, Liu, Ma,
and Pikhurko \cite{GaoLiuMaPikhurko} established phase transitions
in a broad class of degenerate Tur\'an problems for degree power sums.  These results connect extremal problems for degree power sums with
shadow inequalities, changes of extremal construction, and the
spectral theory of inclusion matrices.

Brooks and Linz \cite{BrooksLinz} proved the exact quadratic theorem
for intersecting families and a $t$-intersecting version for
sufficiently large $n$.  Wu and Zhang \cite{WuZhang} subsequently
obtained the sharp threshold in \eqref{eq:classical-ekr} for every
$t$.  More generally, Chen, I\v{l}kovi\v{c}, Le\'on, Liu, and Pikhurko
\cite{ChenEtAl} developed a systematic Tur\'an theory for
degree power sums.  At the codegree level, expansion in the polynomials $\binom{x}{j}$
for integer exponents expresses such sums as positive combinations of
counts of collections of $k$-sets sharing a common $(k-1)$-subset.  Zhou and Yuan \cite{ZhouYuan} used this
connection in matching problems and
asked for the maximum codegree power sum of a $t$-intersecting family
\cite[Problem~4.4]{ZhouYuan}.

At the codegree level, our first application treats arbitrary
$t$-intersection, allows every real $p\geq2$, reaches the exact
Wilson threshold, and determines all boundary equality cases.  Our
second application treats every nontrivial degree level for ordinary
intersecting families at the sharp threshold $n\geq2k$.

For bounded matching number, Brooks--Linz \cite{BrooksLinz} treated
the quadratic objective for sufficiently large $n$, while
Wang--Peng \cite{WangPeng} obtained an explicit cubic threshold for
the related problem of counting unordered pairs $\{F,F'\}$ with
$|F\cap F'|=k-1$.  Recent work of Zhou--Yuan
\cite{ZhouYuan} and our earlier paper \cite{ZhangCaoLu} determined
integer codegree powers when $n$ is sufficiently large; our earlier
paper obtained a linear threshold in uniformity three.  The third
application proves the result for arbitrary $k$ in Frankl's linear
range $n\geq(2s+1)k-s$ and simultaneously removes the integrality
restriction on the exponent.

Our main purpose is to transfer sharp bounds for degree sums and sums of squared degrees to
arbitrary real degree powers without losing the extremal structure.
Two obstacles arise.  A nonintegral power has no finite expansion in the polynomials
$\binom{x}{j}$, while the elementary estimate
\[
  \sum_R d_{\cF}(R)^p
  \leq
  \bigl(\max_R d_{\cF}(R)\bigr)^{p-2}
  \sum_R d_{\cF}(R)^2
\]
usually loses equality when the expected extremal degree sequence
has more than one positive value, so this estimate need not give
the sharp bound at the stated threshold.

Our approach uses convexity, interpolation, and bounds for sums of
powers \cite{ShakedShanthikumar,KarlinStudden,deBoor}. We choose an
upper bound for $x^p$ that agrees with $x^p$ at every degree of the
proposed extremal family. The secant inequalities and Hermite
interpolation used here are classical; our contribution is to combine
them with sharp combinatorial bounds in a way that also determines
all equality cases for real $p$.

The combinatorial part supplies an upper bound for each degree and
bounds for the degree sum and the sum of squared degrees, retaining
their dependence on the family size when necessary. In the matching
problem, the sum of squared degrees is replaced by
$\sum_R(d_{\cF}(R)-s)_+$. The analytic part uses these bounds to
obtain a sharp inequality for the sum of $p$-th powers.

We use three inequalities. A full $t$-star has codegrees in
$\{0,1,D\}$, where $D=n-k+1$. The quadratic bound in
Lemma~\ref{lem:endpoint-secant} agrees with $x^p$ at these values
and holds for all integers $0\leq x\leq D$. For general $r$, the
$r$-degrees of a full star take two positive values. The degree sum
and the bound for the sum of squares determine numbers $z,E$ and
weights $\theta,N-\theta$. The quadratic polynomial that agrees
with $x^p$ and its derivative at $z$ and with $x^p$ at $E$ bounds
$x^p$ from above on $[0,E]$. For bounded matching number, we use a
continuous function that is linear on $[0,s]$ and $[s,D]$ and agrees
with $x^p$ at $0,s,D$. A bound for $\sum_R(d_{\cF}(R)-s)_+$ then
gives the required inequality for every real $p\geq1$.

The central result, Lemma~\ref{lem:two-moment-envelope}, states the
consequence of the quadratic interpolation argument for arbitrary
nonnegative real sequences. The applications follow the same steps:
compute the degrees of the proposed extremal family, prove the
required bounds for the sums, and apply the corresponding inequality
for $x^p$. Equality in that inequality and in the combinatorial bounds
then determines the extremal family.

Throughout, we interpret a binomial coefficient as zero if its lower argument
is negative or exceeds its upper argument.  The nonzero codegrees of
a full $t$-star have the following distribution:
\begin{equation}\label{eq:star-codegree-distribution}
 \begin{array}{c|c}
 \text{codegree}&\text{number of $(k-1)$-sets}\\ \hline
 n-k+1&\displaystyle \binom{n-t}{k-t-1}\\[4pt]
 1&\displaystyle t\binom{n-t}{k-t}.
 \end{array}
\end{equation}
Consequently,
\begin{equation}\label{eq:star-codegree-value}
  \co_p(\cS_T)
  =
  \binom{n-t}{k-t}
  \bigl(t+(k-t)(n-k+1)^{p-1}\bigr).
\end{equation}

Our first application shows that the classical
sharp threshold remains valid for every real exponent $p\geq2$.

\begin{theorem}\label{thm:codegree-main}
Let $1\leq t\leq k\leq n$ with $k\geq2$,  $p\geq2$ be real
and suppose that
$
  n\geq(t+1)(k-t+1).
$
If $\cF\subseteq\binom{[n]}k$ is $t$-intersecting, then
\begin{equation}\label{eq:codegree-main-bound}
  \co_p(\cF)
  \leq  \co_p(\cS_T)=
  \binom{n-t}{k-t}
  \bigl(t+(k-t)(n-k+1)^{p-1}\bigr).
\end{equation}

If $p>2$, equality holds if and only if $\cF$ is a full
$t$-star.  If $p=2$ and $k>t$, equality holds precisely for the
following families, up to permutation:
\begin{enumerate}[label=\textup{(\roman*)}]
  \item a full $t$-star;
  \item if $t=1$ and $n=2k$, the family
  $\binom Zk$ of all $k$-subsets of a fixed
  $Z\in\binom{[n]}{2k-1}$;
  \item if $t\geq2$, $k=t+1$, and $n=2t+2$, the family
  $\binom Zk$ of all $k$-subsets of a fixed
  $Z\in\binom{[n]}{k+1}$.
\end{enumerate}
When $p=2$ and $k=t$, equality holds if and only if $\cF$
consists of one $k$-set.
\end{theorem}

\begin{remark}
The specialization $p=2$ of Theorem~\ref{thm:codegree-main}
recovers the upper bound of Wu and Zhang \cite{WuZhang}, with the
same sharp threshold, and the theorem extends it to every real
$p\geq2$. Their proof uses the generating sets introduced by
Ahlswede and Khachatrian. We use their sharp estimate for the number
of unordered pairs $\{F,F'\}$ satisfying $|F\cap F'|=k-1$.
Together with the cardinality bound, it bounds
$\sum_R d_{\cF}(R)$ and
$\sum_R d_{\cF}(R)(d_{\cF}(R)-1)$.
Lemma~\ref{lem:endpoint-secant} then gives the result for every
real $p\geq2$.

The equality analysis also determines all equality cases when
$n=(t+1)(k-t+1)$.
For $t=1$ and $n=2k$, Brooks and Linz \cite{BrooksLinz} had already
shown that both a full $1$-star and the family of all $k$-subsets
of a fixed $(2k-1)$-set are extremal.  For $t\geq2$, a further
equality family occurs when $k=t+1$ and $n=2t+2$: the family of all
$k$-subsets of a fixed $(t+2)$-set and a full $t$-star both have
codegree squared sum $2(t+1)(t+2)$.  Thus, at the boundary, the assertion that only full $t$-stars attain equality in \cite[Theorem~1.3]{WuZhang}
overlooks these additional extremal families.  Its
equality statement must therefore be supplemented by the families in
Theorem~\ref{thm:codegree-main}(ii) and (iii), with case (ii)
previously known.
\end{remark}

\begin{remark}
Problem~4.4 of Zhou and Yuan \cite{ZhouYuan} is stated without restricting $n$ to
the range $n\geq(t+1)(k-t+1)$. Theorem~\ref{thm:codegree-main} therefore answers the
problem throughout the sharp Erd\H{o}s--Ko--Rado range, but not for
all $n$.  Below this range the correct extremal family may be a
different construction from the complete intersection theorem, and
the problem for arbitrary $n$ remains open.
\end{remark}

Our second application is not restricted to codegrees.  At a general
degree level,
\[
  \ell_{r,2}(\cF)
  =
  \sum_{F,F'\in\cF}\binom{|F\cap F'|}{r}.
\]
Thus pairs whose intersection has any size at least $r$ contribute
to this sum, so the bound counting only pairs with intersection size
$k-1$ is no longer sufficient. Bey's inequality bounds
$\ell_{r,2}(\cF)$ in terms of $|\cF|$, and
Lemma~\ref{lem:two-moment-envelope} gives the bound for $p\geq2$.
For $x\in[n]$, write
\[
  \cS_x:=\left\{F\in\binom{[n]}k:x\in F\right\}.
\]
We call $\cS_x$ a {\em full star}.  The other construction used below when $n=2k$ is the family $\binom Zk$, where
$Z\in\binom{[n]}{2k-1}$.

\begin{theorem}\label{thm:all-level-main}
Let $k\geq2$, $n\geq2k$, $1\leq r\leq k-1$, and let $p\geq2$
be real.  If
$\cF\subseteq\binom{[n]}k$ is intersecting, then
\begin{equation}\label{eq:all-level-main}
  \ell_{r,p}(\cF)
  \leq
  \ell_{r,p}(\cS_x)
  =
  \binom{n-1}{r-1}\binom{n-r}{k-r}^{\!p}
  +
  \binom{n-1}{r}\binom{n-r-1}{k-r-1}^{\!p}.
\end{equation}
If $p>2$ or $p=2$ and $n>2k$, equality holds if and only if $\cF$ is a full
star.    If $p=2$ and $n=2k$, equality holds
if and only if $\cF$ is a full star or
$\cF=\binom Zk$ for some $Z\in\binom{[2k]}{2k-1}$.
\end{theorem}

Our third and final main application concerns a bounded matching
number.  A \emph{matching} in a family
\(\cF\subseteq\binom{[n]}k\) is a collection of pairwise disjoint
members of \(\cF\).  The \emph{matching number} \(\nu(\cF)\) is the
maximum size of a matching contained in \(\cF\).  Thus,
\(\nu(\cF)\le s\) precisely when \(\cF\) contains no \(s+1\)
pairwise disjoint members. For an
$s$-set $S\subseteq[n]$, define
\[
  \cB_S
  :=
  \left\{F\in\binom{[n]}k:F\cap S\neq\emptyset\right\}.
\]
Thus $\nu(\cB_S)=s$ in the range below.  Its codegree distribution
has two positive levels: a $(k-1)$-set meeting $S$ has codegree
$D=n-k+1$, while a $(k-1)$-set disjoint from $S$ has codegree $s$.

\begin{theorem}\label{thm:matching-main}
Let $k\geq2$, $s\geq1$, and
$
  n\geq(2s+1)k-s.
$
If $\cF\subseteq\binom{[n]}k$ satisfies $\nu(\cF)\leq s$, then,
for every real $p\geq1$,
\begin{equation*}
\begin{split}
  \co_p(\cF)
  \leq \co_p(\cB_S)
  ={}
  \left(\binom n{k-1}-\binom{n-s}{k-1}\right)
  (n-k+1)^p+\binom{n-s}{k-1}s^p.
\end{split}
\end{equation*}
Equality holds if and only if $\cF=\cB_S$ up to permutation.
\end{theorem}

The threshold in Theorem~\ref{thm:matching-main} is linear in $sk$
and is exactly the range of Frankl's matching theorem
\cite{FranklMatching}.  The new combinatorial bound is the sharp estimate for
$\sum_R(d_{\cF}(R)-s)_+$ in \eqref{eq:matching-ce} proved in
Section~\ref{sec:matching}; its sharp quadratic consequence is
recorded in Corollary~\ref{cor:matching-quadratic}.  Lemma~\ref{lem:hinge-transfer} then treats nonintegral powers
without an expansion in the polynomials $\binom{x}{j}$.  For comparison, the explicit threshold for arbitrary $k$ of Wang--Peng \cite{WangPeng} for counting unordered pairs $\{F,F'\}$ with
$|F\cap F'|=k-1$ is, after reindexing, $n\geq2(s+1)k^3$, while
\cite{ZhouYuan,ZhangCaoLu} treat integer powers for sufficiently
large $n$, and our earlier paper \cite{ZhangCaoLu} gives an explicit
linear threshold when $k=3$. Thus
Theorem~\ref{thm:matching-main} supplies both the linear range for
general $k$ and all real exponents $p\geq1$.  More precisely, it
strengthens \cite[Theorem~1.9]{ZhangCaoLu}: the exponent is extended
from positive integers to every real $p\geq1$, and the unspecified threshold for sufficiently large $n$ is replaced by Frankl's explicit linear
range.  The theorem counting collections of sets with the same pairwise
intersection and the stability
results in \cite{ZhangCaoLu} are separate and are not implied by the
present paper.

The same analytic inequalities also apply to other finite incidence
relations. In Subsection~\ref{subsec:other-incidence-structures}, we
combine known intersection theorems and bounds for sums of squared
degrees in finite vector spaces with these inequalities to obtain
bounds for degree power sums of families of subspaces.

Section~\ref{sec:two-moment} proves the four analytic lemmas and
states the combinatorial bounds used in the applications. The
applications use, respectively, the estimate of Wu and Zhang for
pairs with intersection size $k-1$, Bey's inequality for sums of
squared degrees, and the bound for $\Phi_s(\cF)$ defined in
Section~\ref{sec:matching}. Section~\ref{sec:all-level-quadratic}
determines all equality cases when $n=2k$ using Bey's equality classification.
Finally, Section~\ref{sec:further-work} treats families of subspaces
and formulates the corresponding problems for general $r$-degrees,
specifying the additional combinatorial inequalities needed.

\section{Convex inequalities for sums of powers}
\label{sec:two-moment}

We first state the cardinality and quadratic bounds used in the
intersection applications. We then prove four analytic lemmas for
real powers. The bound for $\sum_R(d_{\cF}(R)-s)_+$ is proved in
Section~\ref{sec:matching}. The analytic lemmas apply to arbitrary
sequences satisfying their stated assumptions.

\subsection{Combinatorial bounds}

We begin with the precise part of the classical intersection theory
needed below.  The inequality and the uniqueness statement when $n>(t+1)(m-t+1)$
are due to Wilson \cite{Wilson}, while the boundary classification for
$t\geq2$ follows from the complete intersection theorem
\cite{AK}.

\begin{theorem}[Wilson, Ahlswede and Khachatrian]\label{thm:classical}
Let $1\leq t\leq m\leq n$, and let
$\cF\subseteq\binom{[n]}m$ be $t$-intersecting.  If
$
  n\geq(t+1)(m-t+1),
$
then
\[
  |\cF|\leq\binom{n-t}{m-t}.
\]
If the inequality on $n$ is strict, equality holds only for a full
$t$-star.  If $t\geq2$, $t<m$, and
$n=(t+1)(m-t+1)$, equality holds precisely for a full $t$-star or
for a family
\[
  \left\{F\in\binom{[n]}m:|F\cap Z|\geq t+1\right\},
  \qquad |Z|=t+2,
\]
up to permutation.
\end{theorem}

We shall also use Bey's bound for the sum of squared degrees of an
arbitrary uniform family, including its equality classification \cite{Bey}.

\begin{theorem}[Bey]\label{thm:bey}
If $\cF\subseteq\binom{[n]}k$, $1\leq r\leq k-1$, and
$m=|\cF|$, then
\begin{equation}\label{eq:bey}
\begin{split}
  \ell_{r,2}(\cF)
  \leq{}&
  \frac{\binom kr\binom{k-1}r}{\binom{n-1}r}\,m^2+
  \binom{k-1}{r-1}
  \binom{n-r-1}{k-r}\,m.
\end{split}
\end{equation}
If $n\geq k+2$, equality holds precisely for the empty family,
the complete family $\binom{[n]}k$, a full star, or
$\binom{[n]\setminus\{x\}}k$ for some $x\in[n]$.
\end{theorem}

For $\cF\subseteq\binom{[n]}k$, define
\[
  M_1(\cF):=
  \sum_{E\in\binom{[n]}{k-1}}d_{\cF}(E),
  \qquad
  M_2(\cF):=
  \sum_{E\in\binom{[n]}{k-1}}
  d_{\cF}(E)\bigl(d_{\cF}(E)-1\bigr).
\]
Double counting gives
\begin{equation}\label{eq:M1-identity}
  M_1(\cF)=k|\cF|.
\end{equation}
Let
\[
  \zeta_{k-1}(\cF)
  :=
  \left|\left\{\{F,F'\}\in\binom{\cF}{2}:
  |F\cap F'|=k-1\right\}\right|.
\]
Every unordered pair counted by $\zeta_{k-1}(\cF)$ has a unique
common $(k-1)$-set, and hence
\begin{equation}\label{eq:M2-zeta}
  M_2(\cF)=2\zeta_{k-1}(\cF).
\end{equation}

\begin{lemma}\label{lem:moment-input}
Let $1\leq t<k\leq n$, let $\cF\subseteq\binom{[n]}k$ be
$t$-intersecting, and suppose that
$n\geq(t+1)(k-t+1)$.  Then
\begin{align*}
  M_1(\cF)
  &\leq k\binom{n-t}{k-t}
    =M_1(\cS_T),\\
  M_2(\cF)
  &\leq (k-t)(n-k)\binom{n-t}{k-t}
    =M_2(\cS_T).
\end{align*}
\end{lemma}

\begin{proof}
The first inequality follows from Theorem~\ref{thm:classical} and
\eqref{eq:M1-identity}.  Wu and Zhang
\cite[Theorem~1.4]{WuZhang} proved
\[
  \zeta_{k-1}(\cF)
  \leq
  \frac12(k-t)(n-k)\binom{n-t}{k-t}.
\]
Together with \eqref{eq:M2-zeta}, this gives the second inequality.
Only this inequality, and not the equality statement in
\cite[Theorem~1.4]{WuZhang}, is used here.  For a full $t$-star,
\eqref{eq:star-codegree-distribution} gives
\[
  M_2(\cS_T)
  =
  \binom{n-t}{k-t-1}(n-k+1)(n-k)
  =
  (k-t)(n-k)\binom{n-t}{k-t}.
\]
\end{proof}

\subsection{Analytic inequalities}

The next two lemmas give quadratic polynomials that bound $x^p$
from above. The third uses the Hermite polynomial to bound sums of
powers from bounds for $\sum_i x_i$ and $\sum_i x_i^2$. The final
lemma uses a continuous function that is linear on $[0,L]$ and
$[L,D]$. We begin with the inequality used for codegrees.

\begin{lemma}\label{lem:endpoint-secant}
Let $p\geq2$ be real and  $D\geq2$ be an integer.  Put
\[
  \sigma_{p,D}:=\frac{D^{p-1}-1}{D-1}.
\]
Then for  every integer $0\leq d\leq D$, we have
\begin{equation}\label{eq:endpoint-secant}
  d^p\leq d+\sigma_{p,D}d(d-1).
\end{equation}
If $p>2$, equality holds precisely for
$d\in\{0,1,D\}$.
\end{lemma}

\begin{proof}
For $d\geq2$, divide \eqref{eq:endpoint-secant} by $d(d-1)$.
The assertion becomes
\[
  \frac{d^{p-1}-1}{d-1}
  \leq
  \frac{D^{p-1}-1}{D-1}.
\]
The left-hand side is the secant slope of the convex function
$x\mapsto x^{p-1}$ between $1$ and $d$, and is therefore
nondecreasing in $d$.  It is strictly increasing when $p>2$.
The cases $d=0,1$ are identities.
\end{proof}

For general $r$, the $r$-degrees of a full star take two positive
values. The polynomial in Lemma~\ref{lem:endpoint-secant} need not
agree with $x^p$ at both values. We therefore use a quadratic
polynomial that agrees with $x^p$ and its derivative at $z$ and with
$x^p$ at $E$. The resulting bound holds for every real $x\in[0,E]$.

\begin{lemma}\label{lem:hermite-majorant}
Let $p\geq2$ and $0\leq z<E$.  Define
\[
  \Gamma_{z,E}
  :=
  \frac{E^p-z^p-pz^{p-1}(E-z)}{(E-z)^2}
\]
and
\[
  H_{z,E}(x)
  :=
  z^p+pz^{p-1}(x-z)+\Gamma_{z,E}(x-z)^2.
\]
Then $\Gamma_{z,E}\geq0$ and
\begin{equation}\label{eq:hermite-majorant}
  x^p\leq H_{z,E}(x)
  \qquad (0\leq x\leq E).
\end{equation}
If $p>2$, equality holds exactly at $x\in\{z,E\}$.
\end{lemma}

\begin{proof}
Put $f(x)=x^p$. Its second derivative is continuous and nondecreasing
on $[0,E]$. For $0\leq x\leq E$, Taylor's formula gives
\[
 f(x)=f(z)+f'(z)(x-z)+(x-z)^2K_z(x),
 \qquad
 K_z(x):=\int_0^1(1-u)f''\bigl(z+u(x-z)\bigr)\,du.
\]
Since $x\leq E$ and $f''$ is nondecreasing,
\[
 0\leq K_z(x)\leq K_z(E)=\Gamma_{z,E}.
\]
This proves \eqref{eq:hermite-majorant}. If $p>2$, then $f''$ is
strictly increasing, so $K_z(x)<K_z(E)$ whenever $x<E$.
The factor $(x-z)^2$ therefore shows that equality holds exactly
at $z$ and $E$.
\end{proof}

We next prove an inequality for arbitrary nonnegative real
sequences. The constants $L,D,a,N-a$ describe the expression
$aD^p+(N-a)L^p$ on the right of the desired bound; $a$ and $N-a$
are allowed to be real. In the applications they are the numbers of
degrees equal to $D$ and $L$, respectively.

The parameter $m$ varies over $[0,M]$. At $m=M$, the definitions
below give
\[
  cM=aD+(N-a)L,
  \qquad
  \mathcal B(M)=aD^2+(N-a)L^2.
\]
Thus the bounds for $\sum_i x_i$ and $\sum_i x_i^2$ agree with
these two expressions when $m=M$. We write $\xi=L/M$ and
$\eta=a(D-L)/M$ to simplify the calculations.

\begin{lemma}\label{lem:two-moment-envelope}
Let $N$ be a positive integer, and let $L,D,M,a$ be real numbers
satisfying
\[
0\leq L<D\leq M,
\qquad
0<a<N,
\]
and put
\[
  c:=\frac{aD+(N-a)L}{M},
  \qquad
  \xi:=\frac LM,
  \qquad
  \eta:=c-N\xi.
\]
For $0\leq u\leq M$, define
\[
  \mathcal B(u):=c\xi u^2+D\eta u.
\]
Fix $m\in[0,M]$, and suppose that $x_1,\ldots,x_N$ are nonnegative
real numbers satisfying
\begin{equation}\label{eq:abstract-envelope-assumptions}
  x_i\leq\min\{m,D\},\qquad
  \sum_{i=1}^N x_i=cm,\qquad
  \sum_{i=1}^N x_i^2\leq\mathcal B(m).
\end{equation}
Then, for every real $p\geq2$,
\begin{equation}\label{eq:abstract-envelope}
  \sum_{i=1}^N x_i^p
  \leq
  aD^p+(N-a)L^p.
\end{equation}
If $p>2$, equality forces $m=M$, $a\in\mathbb Z$ and, after
reordering,
\[
  (x_1,\ldots,x_N)
  =
  (\underbrace{D,\ldots,D}_{a},
   \underbrace{L,\ldots,L}_{N-a}).
\]
\end{lemma}

\begin{proof}
	The definitions give
	\[
	\eta=\frac{a(D-L)}M>0,
	\qquad
	cM=aD+(N-a)L<ND.
	\]
	In particular, $0<c<N$. The case $m=0$ is immediate, so assume
	$m>0$. We prove three claims. First, we construct $z,\theta$ so
	that $\theta E+(N-\theta)z=S$ and
	$\theta E^2+(N-\theta)z^2=\widehat{\mathcal B}$, with
	$E,S,\widehat{\mathcal B}$ defined below. These identities give
	an upper bound $\mathcal U_p(m)$ for $\sum_i x_i^p$. We then
	prove that $\mathcal U_p(m)$ is increasing in $m$, evaluate it
	at $m=M$, and determine all equality cases.
	
	Write
	\[
	E:=\min\{m,D\},\qquad
	S:=cm,\qquad
	\widehat{\mathcal B}:=\min\{\mathcal B(m),ES\}.
	\]
	
	\medskip
	\noindent\textbf{Claim 1.}
	There exist $z$ and $\theta$ satisfying
	$0\leq z<E$ and $0\leq\theta\leq N$ such that
	\[
	\theta E+(N-\theta)z=S,
	\qquad
	\theta E^2+(N-\theta)z^2=\widehat{\mathcal B},
	\]
	and
	\[
	\sum_{i=1}^N x_i^p
	\leq
	\theta E^p+(N-\theta)z^p.
	\]
	
	\smallskip
	\noindent\emph{Proof of Claim 1.}
	The assumptions give
	\[
	\sum_{i=1}^N x_i^2\leq\mathcal B(m).
	\]
	Moreover, since $0\leq x_i\leq E$, we have
	$x_i^2\leq Ex_i$, and hence
	\[
	\sum_{i=1}^N x_i^2
	\leq E\sum_{i=1}^N x_i
	=ES.
	\]
	Therefore
	\[
	\sum_{i=1}^N x_i^2
	\leq\min\{\mathcal B(m),ES\}
	=\widehat{\mathcal B}.
	\]
	Put $\rho:=c/N$. Direct calculation gives
	\[
	\mathcal B(m)-\frac{S^2}{N}
	=
	\eta m(D-\rho m)\geq0,
	\]
	because
	\[
	D-\rho M=\frac{(N-a)(D-L)}N>0.
	\]
	Also $E\geq S/N$: for $m<D$ this follows from $c<N$, while for
	$m\geq D$ it follows from $D>\rho M\geq\rho m$. Thus
	\[
	\frac{S^2}{N}\leq\widehat{\mathcal B}\leq ES.
	\]
	
	Define
	\begin{equation}\label{eq:abstract-z-theta}
		z:=
		\frac{ES-\widehat{\mathcal B}}{EN-S},
		\qquad
		\theta:=
		\frac{S-Nz}{E-z}.
	\end{equation}
	The denominator $EN-S$ is positive: it equals $m(N-c)$ when
	$m<D$, while for $m\geq D$ it is at least
	\[
	DN-cM=(N-a)(D-L)>0.
	\]
	Moreover, $0\leq z<E$ and $0\leq\theta\leq N$. Indeed, if
	$m\geq D$, then
	\[
	DS-\mathcal B(m)=\xi m(DN-cm)\geq0.
	\]
	Since $\mathcal B(m)\leq DS=ES$, we have
	$\widehat{\mathcal B}=\mathcal B(m)$. Hence, using $E=D$ and
	$S=cm$,
	\[
	z
	=
	\frac{DS-\mathcal B(m)}{DN-S}
	=
	\frac{\xi m(DN-cm)}{DN-cm}
	=
	\xi m
	\leq L<D.
	\]
	If $m<D$, then $E=m$.
	If $\widehat{\mathcal B}=mS$, then the definition of $z$ immediately
	gives $z=0$.  Otherwise,
	$\widehat{\mathcal B}=\mathcal B(m)<mS$, and hence
	\[
	\begin{aligned}
		z=\frac{mS-\mathcal B(m)}{mN-S}=
		\frac{c(1-\xi)m-D\eta}{N-c}=
		\xi m-\frac{\eta}{N-c}(D-m),
	\end{aligned}
	\]
	where the last equality uses $\eta=c-N\xi$.  In this second case,
	$\mathcal B(m)<mS$ gives $z>0$, while
	$\eta>0$, $N-c>0$, and $m<D$ give
	\[
	z<\xi m<m.
	\]
	Finally,
	$z\leq S/N$ follows from
	$\widehat{\mathcal B}\geq S^2/N$, and this gives
	$0\leq\theta\leq N$.
	
	The weights $\theta$ and $N-\theta$ sum to $N$, and the
	definitions of $z$ and $\theta$ give
	\begin{equation}\label{eq:abstract-two-point-moments}
		\theta E+(N-\theta)z=S,
		\qquad
		\theta E^2+(N-\theta)z^2=\widehat{\mathcal B}.
	\end{equation}
	
	Write
	\[
	H_{z,E}(x)
	=
	\alpha+\beta x+\Gamma_{z,E}x^2
	\]
	for suitable real numbers $\alpha$ and $\beta$.  By
	Lemma~\ref{lem:hermite-majorant}, we have
	$x_i^p\leq H_{z,E}(x_i)$ for every $i$.  Moreover,
	$\Gamma_{z,E}\geq0$, so the bound $\sum_i x_i^2\leq\widehat{\mathcal B}$ and
	\eqref{eq:abstract-two-point-moments} give
	\begin{equation}\label{eq:abstract-moment-bound}
		\begin{aligned}
			\sum_{i=1}^N x_i^p
			&\leq
			\sum_{i=1}^N H_{z,E}(x_i)\\
			&=
			N\alpha
			+\beta\sum_{i=1}^N x_i
			+\Gamma_{z,E}\sum_{i=1}^N x_i^2\\
			&\leq
			N\alpha+\beta S+\Gamma_{z,E}\widehat{\mathcal B}\\
			&=
			\theta H_{z,E}(E)
			+(N-\theta)H_{z,E}(z)\\
			&=
			\theta E^p+(N-\theta)z^p
			=:\mathcal U_p(m).
		\end{aligned}
	\end{equation}
	Here the penultimate equality follows from
	\eqref{eq:abstract-two-point-moments}, while the last equality uses
	$H_{z,E}(E)=E^p$ and $H_{z,E}(z)=z^p$.
	This proves Claim~1.\q
	
	\medskip
	\noindent\textbf{Claim 2.}
	After setting $\mathcal U_p(0):=0$, the function
	$\mathcal U_p$ is strictly increasing on $[0,M]$.
	
	\smallskip
	\noindent\emph{Proof of Claim 2.}
	First consider $0<m<D$. Since $E=m$ in this range,
	\[
	\widehat{\mathcal B}
	=
	\min\{\mathcal B(m),mS\}.
	\]
	Define
	\[
	m_0:=\frac{D\eta}{c(1-\xi)}.
	\]
	Here $c>0$ and $1-\xi>0$. Moreover,
	\[
	c(1-\xi)-\eta=\xi(N-c)\geq0,
	\]
	so $0<m_0\leq D$, with equality if and only if $\xi=0$.
	Furthermore,
	\[
	\begin{aligned}
		\mathcal B(m)-mS
		=
		m\bigl(D\eta-c(1-\xi)m\bigr)=
		c(1-\xi)m(m_0-m).
	\end{aligned}
	\]
	Consequently,
	\[
	\widehat{\mathcal B}
	=
	\begin{cases}
		mS, & 0<m\leq m_0,\\
		\mathcal B(m), & m_0<m<D,
	\end{cases}
	\]
	where the second interval is empty when $\xi=0$.
	
	On the initial interval, $z=0$ and
	\[
	\theta=\frac{S}{m}=c,
	\qquad
	\mathcal U_p(m)=cm^p,
	\]
	so $\mathcal U_p(m)$ is strictly increasing.
	
	On the complementary interval $m_0<m<D$, put
	
	\[
	t:=\frac zm.
	\]
	Then
	\[
	t
	=
	\frac{\rho(1-\xi)}{1-\rho}
	-
	\frac{D\eta}{(N-c)m},
	\qquad
	0<t<\xi<\rho.
	\]
	Differentiating gives
	\[
	mt'=\frac{D\eta}{(N-c)m},
	\]
	and hence
	\begin{equation}\label{eq:abstract-t-derivative}
		(1-\rho)mt'
		=
		\rho(1-\xi)-(1-\rho)t
		\leq \rho-t,
	\end{equation}
	where the last inequality follows from $t\leq\xi$.
	
	Furthermore,
	\[
	\mathcal U_p(m)
	=
	Nm^p g(t),
	\qquad
	g(t):=
	\frac{\rho-t+(1-\rho)t^p}{1-t}.
	\]
	Let
	\[
	h_p(t):=1-pt^{p-1}+(p-1)t^p.
	\]
	Since
	\[
	h_p'(t)
	=
	-p(p-1)t^{p-2}(1-t)\leq0
	\qquad\text{and}\qquad
	h_p(1)=0,
	\]
	we have $h_p(t)\geq0$ on $[0,1]$. Consequently,
	\[
	g'(t)
	=
	-\frac{(1-\rho)h_p(t)}{(1-t)^2}\leq0.
	\]
	We also have
	\[
	g(t)\geq\frac{\rho-t}{1-t}
	\]
	and
	\[
	h_p(t)\leq p(1-t).
	\]
	The last inequality follows by integrating
	\[
	-h_p'(s)
	=
	p(p-1)s^{p-2}(1-s)
	\]
	and using
	\[
	(p-1)s^{p-2}(1-s)\leq1
	\qquad (0\leq s\leq1).
	\]
	Moreover, the inequality is strict when $t<1$. Using
	\eqref{eq:abstract-t-derivative}, we obtain
	\[
	\begin{split}
		pg(t)+mt'g'(t)
		=
		pg(t)
		-
		\frac{(1-\rho)mt'h_p(t)}{(1-t)^2}\geq
		\frac{\rho-t}{(1-t)^2}
		\bigl(p(1-t)-h_p(t)\bigr)
		>0.
	\end{split}
	\]
	It follows that
	\[
	\mathcal U_p'(m)
	=
	Nm^{p-1}\bigl(pg(t)+mt'g'(t)\bigr)>0.
	\]
	At $m=m_0$, we have $z=0$, hence $t=0$, and
	\[
	Nm_0^p g(0)
	=
	N\rho m_0^p
	=
	cm_0^p.
	\]
	Thus the two formulas for $\mathcal U_p(m)$ agree at $m_0$.
	Together with strict increase on each nonempty branch, this proves
	that $\mathcal U_p$ is strictly increasing throughout the lower
	range.
	
	Now suppose that $D\leq m\leq M$. Here
	\[
	z=\xi m,
	\qquad
	\theta=\frac{\eta m}{D-\xi m}.
	\]
	If $\xi=0$, then
	\[
	\mathcal U_p(m)=cmD^{p-1},
	\]
	which is strictly increasing. If $\xi>0$, put
	\[
	u:=\frac{\xi m}{D}.
	\]
	Since $\xi m\leq L<D$, we have $0<u<1$, and
	\[
	\mathcal U_p(m)
	=
	ND^p u^p
	+
	\frac{\eta D^p}{\xi}
	\frac{u(1-u^p)}{1-u}.
	\]
	Both terms are strictly increasing in $u\in[0,1)$. Indeed, the
	numerator of the derivative of
	\[
	\frac{u(1-u^p)}{1-u}
	\]
	is
	\[
	1-(p+1)u^p+pu^{p+1}.
	\]
	This numerator equals zero at $u=1$, and its derivative is
	\[
	-p(p+1)u^{p-1}(1-u)<0\qquad(0<u<1).
	\]
	Thus the numerator is positive for $0\leq u<1$.
	Since $u$ is strictly increasing in $m$, the function
	$\mathcal U_p(m)$ is strictly increasing on $[D,M]$.
	
	The formulas for $m<D$ and $m\geq D$ agree at $m=D$. Moreover,
	since $\mathcal U_p(m)=cm^p$ for all sufficiently small $m>0$, the
	function $\mathcal U_p$ extends continuously to $m=0$ by setting
	$
	\mathcal U_p(0):=0.
	$
	With this convention, $\mathcal U_p$ is strictly increasing on
	$[0,M]$. This proves Claim~2.\q
	
	\medskip
	\noindent\textbf{Claim 3.}
	At $m=M$ one has
	\[
	\mathcal U_p(M)=aD^p+(N-a)L^p.
	\]
	Consequently, \eqref{eq:abstract-envelope} holds. Moreover, if
	$p>2$, equality in \eqref{eq:abstract-envelope} forces $m=M$ and
	exactly $a$ of the $x_i$ to equal $D$, with all others equal to $L$.
	
	\smallskip
	\noindent\emph{Proof of Claim 3.}
	At $m=M$,
	equations \eqref{eq:abstract-z-theta} give
	\[
	z=L,\qquad \theta=a.
	\]
	Therefore
	\[
	\mathcal U_p(M)=aD^p+(N-a)L^p,
	\]
	which proves \eqref{eq:abstract-envelope}.
	
	Finally, suppose that $p>2$ and equality holds. Strict
	monotonicity forces $m=M$, and every inequality in
	\eqref{eq:abstract-moment-bound} must be an equality. At this
	endpoint the Hermite majorant is exact only at $L$ and $D$, so every
	$x_i$ belongs to $\{L,D\}$. The identity $\sum_i x_i=cM$ then forces
	exactly $a$ of the variables to equal $D$, proving the equality
	statement. This proves Claim~3.\q
	
	By Claims 1 to 3, the conclusion holds.
\end{proof}

In the matching application, the codegrees of $\cB_S$ take two
positive values $L<D$, and shifting gives a bound for
$\sum_i(x_i-L)_+$. The next lemma uses this bound together with
the bound for $\sum_i x_i$.

\begin{lemma}\label{lem:hinge-transfer}
Let $N$ be a positive integer,  and let \(L,D,a\) be real numbers
satisfying $0<L<D$ and
$0\leq a\leq N$.
Suppose that the real numbers $x_1,\ldots,x_N$ satisfy \(0\leq x_i\leq D\) for every \(1\leq i\leq N\), together with
\begin{equation}\label{eq:hinge-assumptions}
  \sum_{i=1}^N x_i\leq aD+(N-a)L,
  \qquad
  \sum_{i=1}^N(x_i-L)_+\leq a(D-L).
\end{equation}
Then, for every real $p\geq1$,
\begin{equation}\label{eq:hinge-transfer}
  \sum_{i=1}^N x_i^p
  \leq aD^p+(N-a)L^p.
\end{equation}
For $p>1$, equality holds in \eqref{eq:hinge-transfer} if and only if
$a\in\mathbb Z$ and, after reordering,
\[
  (x_1,\ldots,x_N)
  =
  (\underbrace{D,\ldots,D}_{a},
   \underbrace{L,\ldots,L}_{N-a}).
\]
\end{lemma}

\begin{proof}
Convexity of $x\mapsto x^p$ gives, for every $0\leq x\leq D$,
\begin{equation}\label{eq:hinge-pointwise}
  x^p
  \leq
  L^{p-1}x
  +
  \left(
    \frac{D^p-L^p}{D-L}-L^{p-1}
  \right)(x-L)_+.
\end{equation}
Indeed, on $[0,L]$, we have $x^p\leq L^{p-1}x$, while on
$[L,D]$ the right-hand side is the chord joining
$(L,L^p)$ and $(D,D^p)$.  The coefficient of $(x-L)_+$ is
nonnegative.  Summing \eqref{eq:hinge-pointwise} and applying
\eqref{eq:hinge-assumptions} yields
\begin{align*}
  \sum_{i=1}^N x_i^p
  &\leq
  L^{p-1}\bigl(aD+(N-a)L\bigr)+
  \left(
    \frac{D^p-L^p}{D-L}-L^{p-1}
  \right)a(D-L)\\
  &=aD^p+(N-a)L^p.
\end{align*}

If $p>1$, the coefficients of $x$ and $(x-L)_+$ in
\eqref{eq:hinge-pointwise} are
positive, and \eqref{eq:hinge-pointwise} is strict away from
$\{0,L,D\}$.  Equality therefore forces equality in both parts of
\eqref{eq:hinge-assumptions} and $x_i\in\{0,L,D\}$ for every $i$.
The identity $\sum_i(x_i-L)_+=a(D-L)$ gives exactly $a$ copies
of $D$. The identity $\sum_i x_i=aD+(N-a)L$ then gives $N-a$
copies of $L$ and no copies of zero.
\end{proof}

\begin{remark}
The numbers $z,\theta$ in \eqref{eq:abstract-z-theta} satisfy
\eqref{eq:abstract-two-point-moments}. Since the coefficient
$\Gamma_{z,E}$ of $x^2$ is nonnegative, substituting these
identities into the sum of $H_{z,E}(x_i)$ gives
\[
  \sum_i x_i^p\leq\theta E^p+(N-\theta)z^p.
\]
This argument allows arbitrary real weights $\theta$ and
$N-\theta$. It does not require a sequence with exactly $\theta$
entries equal to $E$.
\end{remark}

\section{First application: codegree powers}

We first apply Lemma~\ref{lem:endpoint-secant}. The combinatorial
bounds are those in Lemma~\ref{lem:moment-input}: the classical
Erd\H{o}s--Ko--Rado theorem bounds $M_1(\cF)$, and the estimate
of Wu and Zhang for pairs with intersection size $k-1$ bounds
$M_2(\cF)$. Summing the inequality in
Lemma~\ref{lem:endpoint-secant} gives the result for every real
$p\geq2$.

\subsection{The upper bound}

\begin{proof}[Proof of Theorem~\ref{thm:codegree-main}: upper bound]
If $k=t$, two distinct $k$-sets cannot be $t$-intersecting.
Therefore $|\cF|\leq1$ and $\co_p(\cF)\leq k=t$, which is the
right-hand side of \eqref{eq:codegree-main-bound}.  Equality in
this case holds exactly when $\cF$ is a singleton, equivalently a
full $t$-star.  We may thus assume $k>t$.

Put
\[
  D:=n-k+1,\qquad
  \sigma:=\frac{D^{p-1}-1}{D-1}.
\]
Every codegree is an integer between $0$ and $D$.  Summing
\eqref{eq:endpoint-secant} over all $(k-1)$-sets and applying
Lemma~\ref{lem:moment-input} gives
\begin{align}
  \co_p(\cF)
  &\leq
  M_1(\cF)+\sigma M_2(\cF)\notag\\
  &\leq
  M_1(\cS_T)+\sigma M_2(\cS_T).
  \label{eq:main-chain}
\end{align}
The two nonzero codegree levels of $\cS_T$ are $D$ and $1$, so
Lemma~\ref{lem:endpoint-secant} is exact on its degree sequence.
Consequently the last expression in \eqref{eq:main-chain} is
\[
  \binom{n-t}{k-t-1}D^p
  +
  t\binom{n-t}{k-t}
  =
  \co_p(\cS_T).
\]
Now \eqref{eq:star-codegree-value} gives
\eqref{eq:codegree-main-bound}.
\end{proof}

\subsection{The quadratic equality cases}

At $p=2$, the identity $d^2=d+d(d-1)$ gives
\begin{equation}\label{eq:co2-moments}
  \co_2(\cF)
  =
  M_1(\cF)+M_2(\cF)
  \leq
  M_1(\cS_T)+M_2(\cS_T),
\end{equation}
and equality requires equality in both moment bounds.

For $t=1$, the complete equality statement follows from Brooks and
Linz \cite[Theorem~1.3]{BrooksLinz}: above $n=2k$ only a star is
extremal, while at $n=2k$ the two extremal types are a star and the
family $\binom Zk$ of all $k$-subsets of a fixed
$Z\in\binom{[n]}{2k-1}$.
It remains to treat $t\geq2$.

Fix a $(t+2)$-set $Z$ and write
\[
  \cA_Z
  :=
  \left\{F\in\binom{[n]}k:|F\cap Z|\geq t+1\right\}.
\]

\begin{lemma}\label{lem:A1-comparison}
Let $t\geq2$, $k\geq t+1$, and
$n=(t+1)(k-t+1)$.  Then
\begin{equation}\label{eq:A1-difference}
  \co_2(\cS_T)-\co_2(\cA_Z)
  =
  t(t+1)(k-t-1)
  \binom{(t+1)(k-t)-1}{k-t-1}.
\end{equation}
In particular, the two values are equal if and only if $k=t+1$.
\end{lemma}

\begin{proof}
Put
\[
  \nu:=n-t-2=(t+1)(k-t)-1.
\]
For $E\in\binom{[n]}{k-1}$, its codegree in $\cA_Z$ is $n-k+1$
when $|E\cap Z|\geq t+1$, is $2$ when
$|E\cap Z|=t$, and is zero otherwise.  Therefore
\begin{align}
  \co_2(\cA_Z)
  ={}&
  \left((t+2)\binom \nu{k-t-2}+\binom \nu{k-t-3}\right)
  (n-k+1)^2
  +2(t+2)(t+1)\binom \nu{k-t-1}.
  \label{eq:A1-co2}
\end{align}
On the other hand,
\begin{equation}\label{eq:star-co2-boundary}
  \co_2(\cS_T)
  =
  \binom{\nu+2}{k-t-1}(n-k+1)^2
  +t\binom{\nu+2}{k-t}.
\end{equation}
If $k=t+1$, both formulas equal $2(t+1)(t+2)$, and the
right-hand side of \eqref{eq:A1-difference} is zero.  For
$k\geq t+2$, put $C_0=\binom \nu{k-t-1}>0$.  The identities
\begin{align*}
 \binom{\nu+2}{k-t-1}
 &=\frac{(\nu+2)(\nu+1)}{(n-k+1)(n-k)}C_0,
 &
 \binom{\nu+2}{k-t}
 &=\frac{(\nu+2)(\nu+1)}{(k-t)(n-k)}C_0,\\
 \binom \nu{k-t-2}
 &=\frac{k-t-1}{n-k}C_0,
 &
 \binom \nu{k-t-3}
 &=\frac{(k-t-1)(k-t-2)}{(n-k+1)(n-k)}C_0
\end{align*}
and direct substitution into \eqref{eq:A1-co2} and
\eqref{eq:star-co2-boundary} give
\[
  \frac{\co_2(\cS_T)-\co_2(\cA_Z)}{C_0}
  =
  t(t+1)(k-t-1),
\]
as required.
\end{proof}

\begin{proof}[Completion of the $p=2$ equality classification]
Assume $t\geq2$ and $k>t$.  If
$n>(t+1)(k-t+1)$, equality in \eqref{eq:co2-moments} forces
$M_1(\cF)=M_1(\cS_T)$.  The equality statement in Theorem~\ref{thm:classical}, applied
with $n>(t+1)(k-t+1)$, gives $\cF=\cS_T$ up to permutation.

Now let $n=(t+1)(k-t+1)$.  Equality again forces maximum
cardinality.  By Theorem~\ref{thm:classical} at $n=(t+1)(k-t+1)$,
$\cF=\cS_T$ or $\cF=\cA_Z$ up to permutation.  Lemma~\ref{lem:A1-comparison} shows that
the latter has the same codegree squared sum exactly when $k=t+1$.
In this case $k=t+1$ and
$\cA_Z=\binom Zk$ is the family of all $k$-subsets of the fixed
$(k+1)$-set $Z$.  Together with
the Brooks--Linz theorem for $t=1$ and the direct case $k=t$, this
proves the quadratic equality assertions in
Theorem~\ref{thm:codegree-main}.
\end{proof}

\subsection{Equality for \texorpdfstring{$p>2$}{p > 2}}

\begin{proof}[Completion of the proof of
Theorem~\ref{thm:codegree-main}]
Assume $p>2$ and $k>t$.  Equality in
\eqref{eq:main-chain} implies equality in both moment bounds, and
hence equality in \eqref{eq:co2-moments}.  The quadratic
classification just proved leaves only a full $t$-star and the two
possible boundary families.

For the family $\binom Zk$ on a $(2k-1)$-set when $t=1$ and
$n=2k$, every positive codegree equals $k$, whereas the largest
possible codegree is $D=k+1$.  For the family $\binom Zk$ on a
$(t+2)$-set when
$k=t+1$ and $n=2t+2$, every positive codegree equals $2$, whereas
$D=t+2$.  In both cases the positive codegree lies strictly between
$1$ and $D$.  Lemma~\ref{lem:endpoint-secant} is therefore strict
on that degree, so neither boundary family can give equality when
$p>2$.
A full $t$-star has only the codegrees $0,1,D$ and attains equality
throughout.  This completes the proof.
\end{proof}

\section{Second application: arbitrary degree levels for
\texorpdfstring{$p\geq2$}{p >= 2}}
\label{sec:all-level-quadratic}

The second application combines Lemma~\ref{lem:two-moment-envelope}
with Bey's inequality. We first record the quadratic consequence,
including all equality cases. For $r=k-1$, this is the theorem of
Brooks and Linz \cite[Theorem~1.3]{BrooksLinz}; the same argument,
using the equality statement in Theorem~\ref{thm:bey}, applies to
every $1\leq r\leq k-1$.

\begin{proposition}\label{prop:all-level-quadratic}
Let $k\geq2$, $n\geq2k$, and $1\leq r\leq k-1$.  If
$\cF\subseteq\binom{[n]}k$ is intersecting, then
\begin{equation}\label{eq:all-level-quadratic}
  \ell_{r,2}(\cF)
  \leq
  \ell_{r,2}(\cS_x)
  =
  \binom{n-1}{r-1}\binom{n-r}{k-r}^{\!2}
  +
  \binom{n-1}{r}\binom{n-r-1}{k-r-1}^{\!2}.
\end{equation}
If $n>2k$, equality holds if and only if $\cF$ is a full
star.  If $n=2k$, equality holds if and only if $\cF$ is a
full star or $\cF=\binom Zk$ for some
$Z\in\binom{[2k]}{2k-1}$.
\end{proposition}

\begin{proof}
Put
\[
  m:=|\cF|,
  \qquad
  m_\star:=\binom{n-1}{k-1}.
\]
The Erd\H{o}s--Ko--Rado theorem gives $m\leq m_\star$.  Denote the
right-hand side of Bey's inequality \eqref{eq:bey}, viewed as a
function of $m$, by $B_{n,k,r}(m)$.  Both its linear coefficient
and its quadratic coefficient are positive.  Hence
\[
  \ell_{r,2}(\cF)
  \leq B_{n,k,r}(m)
  \leq B_{n,k,r}(m_\star).
\]
The $r$-degrees of a full star $\cS_x$ are
\[
  d_{\mathrm{in}}:=\binom{n-r}{k-r}
  \quad\text{on the }\binom{n-1}{r-1}
  \text{ $r$-sets containing }x,
\]
and
\[
  d_{\mathrm{out}}:=\binom{n-r-1}{k-r-1}
  \quad\text{on the }\binom{n-1}{r}
  \text{ $r$-sets avoiding }x.
\]
Substituting $m_\star=\binom{n-1}{k-1}$ in
\eqref{eq:bey} and simplifying gives
\begin{align*}
  B_{n,k,r}(m_\star)
  &=
  \binom{n-1}{r-1}d_{\mathrm{in}}^2
  +
  \binom{n-1}{r}d_{\mathrm{out}}^2
  =
  \ell_{r,2}(\cS_x).
\end{align*}
This proves \eqref{eq:all-level-quadratic}.

Since $B_{n,k,r}(m)$ is strictly increasing, equality forces
$m=m_\star$.  If $n>2k$, the equality statement in the
Erd\H{o}s--Ko--Rado theorem yields $\cF=\cS_x$, up to permutation.

If $n=2k$, equality also requires equality in Bey's inequality.
Since $n=2k\geq k+2$, Theorem~\ref{thm:bey} applies with its
stated equality classification. The empty and complete families
are excluded by
\[
  |\cF|=m_\star=\tfrac12\binom{2k}k.
\]
Thus $\cF$ is a full star or $\binom{[2k]\setminus\{x\}}k$
for some $x\in[2k]$. Both families are intersecting, have size
$m_\star$, and attain equality in Bey's inequality. This proves
the boundary classification.
\end{proof}

With the sharp quadratic input and its equality cases now available,
the result for every real $p\geq2$ follows directly from
Lemma~\ref{lem:two-moment-envelope}.

\begin{proof}[Proof of Theorem~\ref{thm:all-level-main}]
Put $m:=|\cF|$.  The Erd\H{o}s--Ko--Rado theorem gives
\[
  m\leq\binom{n-1}{k-1}.
\]
The collection
$\bigl(d_{\cF}(R)\bigr)_{R\in\binom{[n]}r}$ satisfies
\[
  0\leq d_{\cF}(R)\leq
  \min\left\{m,\binom{n-r}{k-r}\right\}
\]
for every $R\in\binom{[n]}r$, and double counting gives
\[
  \sum_{R\in\binom{[n]}r}d_{\cF}(R)=\binom krm.
\]
Apply Lemma~\ref{lem:two-moment-envelope} with
\[
\begin{gathered}
  N=\binom nr,\quad
  a=\binom{n-1}{r-1},\quad
  D=\binom{n-r}{k-r},\quad
  L=\binom{n-r-1}{k-r-1},\quad
  M=\binom{n-1}{k-1}.
\end{gathered}
\]
These parameters satisfy $0<L<D\leq M$ and $0<a<N$.  To see
$D\leq M$, fix an $r$-set $R$ and a point $x\in R$: every $k$-set
containing $R$ also contains $x$, so the $D$ such sets form a
subfamily of the $M$ $k$-sets containing $x$.  The identity
\[
  \binom kr M=aD+(N-a)L
\]
shows that the parameter $c$ in Lemma~\ref{lem:two-moment-envelope} is $\binom kr$.
Moreover,
\[
  \frac LM=\frac{\binom{k-1}r}{\binom{n-1}r},
  \qquad
  D\left(\binom kr-N\frac LM\right)
  =
  \binom{k-1}{r-1}\binom{n-r-1}{k-r}.
\]
Thus Bey's inequality \eqref{eq:bey} is precisely the bound for $\sum_i x_i^2$ in \eqref{eq:abstract-envelope-assumptions}.
Lemma~\ref{lem:two-moment-envelope} now gives
\[
  \ell_{r,p}(\cF)
  \leq
  \binom{n-1}{r-1}\binom{n-r}{k-r}^{p}
  +
  \binom{n-1}{r}\binom{n-r-1}{k-r-1}^{p},
\]
which is \eqref{eq:all-level-main}.

For $p=2$, the equality statement is
Proposition~\ref{prop:all-level-quadratic}.  Suppose now that
$p>2$.  Equality in Lemma~\ref{lem:two-moment-envelope} forces
$m=\binom{n-1}{k-1}$ and forces the degree multiset to consist of
$a$ copies of $D$ and $N-a$ copies of $L$.  In particular, equality
holds in Bey's quadratic estimate.  Thus
Proposition~\ref{prop:all-level-quadratic} leaves a full star,
or, when $n=2k$, a family $\binom Zk$ on a fixed $(2k-1)$-set.
The latter family has zero $r$-degrees on every $r$-set meeting
$[2k]\setminus Z$, whereas $L>0$.  Its degree multiset is therefore
not the one forced by Lemma~\ref{lem:two-moment-envelope}.  Hence only a
full star can attain equality.
\end{proof}

\section{Third application: bounded matching number}
\label{sec:matching}

We finally apply Lemma~\ref{lem:hinge-transfer} to the codegrees of a family with
bounded matching number.  Fix integers \(k\geq2\) and \(s\geq1\), and
put
\[
  N:=\binom n{k-1},
  \qquad
  D:=n-k+1,
  \qquad
  a_s:=N-\binom{n-s}{k-1},
  \qquad
  b_s:=\binom{n-s}{k-1}.
\]
For a fixed \(s\)-set \(S\), the family \(\cB_S\) has \(a_s\)
codegrees equal to \(D\) and \(b_s\) codegrees equal to \(s\).
Consequently,
\begin{equation}\label{eq:matching-codegrees}
  \co_p(\cB_S)=a_sD^p+b_ss^p.
\end{equation}
Double counting also gives
\begin{equation}\label{eq:matching-incidence-identity}
  k\left(\binom nk-\binom{n-s}k\right)
  =a_sD+b_ss.
\end{equation}

We shall use the following form of Frankl's matching theorem in a
linear range \cite[Theorem~1.1]{FranklMatching}. This result holds
for every integer $s\geq1$; no assumption that $s$ is sufficiently
large is needed.

\begin{theorem}[Frankl]
\label{thm:frankl-matching}
Let \(\cF\subseteq\binom{[n]}k\) satisfy \(\nu(\cF)\leq s\).  If
\[
  n\geq(2s+1)k-s,
\]
then
\begin{equation}\label{eq:frankl-matching}
  |\cF|
  \leq
  \binom nk-\binom{n-s}k.
\end{equation}
Equality holds if and only if \(\cF=\cB_S\) up to permutation.
\end{theorem}

Frankl states the result for $\nu(\cF)=s$.  To obtain the displayed
form, extend $\cF$ to a family maximal under inclusion
$\cG\supseteq\cF$ satisfying $\nu(\cG)\leq s$.  The family $\binom{[n]}k$ has a matching of size $s+1$, since
\[
  n-(s+1)k\geq s(k-1)\geq0,
\]
so $\cG\neq\binom{[n]}k$.  If $\nu(\cG)\leq s-1$, then any missing
$k$-set could be added to $\cG$ without creating a matching of size
$s+1$, contradicting maximality.  Hence $\nu(\cG)=s$, and Frankl's
theorem applied to $\cG$ gives the bound for $\cF$.  If $\cF$
attains equality, the extension cannot be strict; thus $\cF=\cG$ and
Frankl's equality classification is preserved.

The new ingredient is a bound for $\Phi_s(\cF)$, defined below.  We first
record three auxiliary facts used in its proof.

\subsection{Shifting and two auxiliary inequalities}
For a family \(\cF\subseteq\binom{[n]}k\), \(1\leq i<j\leq n\),
and \(A\in\cF\), define
\[
S_{ij}^{\cF}(A):=
\begin{cases}
	(A\setminus\{j\})\cup\{i\},
	&\text{if \(j\in A\), \(i\notin A\), and
		\((A\setminus\{j\})\cup\{i\}\notin\cF\),}\\
	A, &\text{otherwise}.
\end{cases}
\]
The family
\[
S_{ij}\cF:=\left\{S_{ij}^{\cF}(A):A\in\cF\right\}
\]
is called the \emph{\((i,j)\)-shift} of \(\cF\).  A family \(\cF\)
is called \emph{shifted} if
\[
S_{ij}\cF=\cF
\qquad\text{for every }1\leq i<j\leq n.
\]
The procedure of repeatedly applying shifts that change the current
family until no such shift remains is called \emph{shifting}.  This
procedure terminates after finitely many steps, and the resulting
family is shifted.

For a family \(\cF\subseteq\binom{[n]}k\), define
\[
  \Phi_s(\cF)
  :=
  \sum_{R\in\binom{[n]}{k-1}}
  \bigl(d_{\cF}(R)-s\bigr)_+.
\]

\begin{lemma}\label{lem:matching-shift}
Let $k\geq2$, $s\geq1$, and $\cF\subseteq\binom{[n]}k$.
For every \(1\le i<j\le n\),
\[
  \Phi_s(S_{ij}\cF)\geq\Phi_s(\cF).
\]
Moreover, shifting does not increase the matching number.
\end{lemma}

\begin{proof}
The assertion about the matching number is the standard shifting
property.  To prove the first assertion, put
\(\varphi(x):=(x-s)_+\).  For
\(A\in\binom{[n]\setminus\{i,j\}}{k-2}\), define
\begin{align*}
  X_i&:=
  \{x\notin A\cup\{i,j\}:A\cup\{i,x\}\in\cF\},\\
  X_j&:=
  \{x\notin A\cup\{i,j\}:A\cup\{j,x\}\in\cF\},
\end{align*}
and let
\(\varepsilon=\ind_{\{A\cup\{i,j\}\in\cF\}}\).
Before the shift, the codegrees of
\(A\cup\{i\}\) and \(A\cup\{j\}\) are
\[
  |X_i|+\varepsilon
  \quad\text{and}\quad
  |X_j|+\varepsilon.
\]
After the shift, they are
\[
  |X_i\cup X_j|+\varepsilon
  \quad\text{and}\quad
  |X_i\cap X_j|+\varepsilon.
\]
The latter pair has the same sum as the former, and
$|X_i\cup X_j|\geq\max\{|X_i|,|X_j|\}$.  Since \(\varphi\) is convex,
\[
\begin{split}
  \varphi(|X_i\cup X_j|+\varepsilon)
   +\varphi(|X_i\cap X_j|+\varepsilon)\geq
   \varphi(|X_i|+\varepsilon)
   +\varphi(|X_j|+\varepsilon).
\end{split}
\]
All changes in the \((k-1)\)-codegrees are partitioned into these
pairs; codegrees indexed by sets containing both or neither of
\(i,j\) are unchanged.  Summing over \(A\) proves the lemma.
\end{proof}

For an $r$-uniform family
$\mathcal H\subseteq\binom Zr$, define its \emph{ shadow} by
\[
\partial\mathcal H
:=
\left\{
A\in\binom Z{r-1}:
A\subseteq H\text{ for some }H\in\mathcal H
\right\}.
\]
Frankl~\cite[p.~1072]{FranklMatching} observed the improved
coefficient in the shadow inequality below while considering members
of a shifted family that avoid an initial interval. We state and prove
the general inequality because this coefficient is needed here.

\begin{lemma}\label{lem:strong-shadow}
Let \(r\geq2\) and \(s\geq1\) be integers, let
\(\mathcal H\subseteq\binom Zr\), and suppose that
\(\nu(\partial\mathcal H)\leq s\).  Then
\begin{equation*}
  r|\mathcal H|
  \leq
  \bigl((r-1)s-1\bigr)|\partial\mathcal H|.
\end{equation*}
\end{lemma}

\begin{proof}
By the standard compression relations
\[
  \partial(S_{ij}\mathcal H)
  \subseteq S_{ij}(\partial\mathcal H),
\]
the shifting preserves \(|\mathcal H|\), does not increase the size or
the matching number of its shadow, and makes the shadow shifted.
It is therefore enough to prove the assertion for a shifted family.
We use induction on $r$ and, for each fixed $r$, induction on
$v:=|Z|$, after identifying $Z$ with $[v]$.

For \(r=2\), the shadow is $\{\{v\}:v\in\bigcup_{H\in\mathcal H}H\}$.
Writing \(u:=|\partial\mathcal H|\), the hypothesis gives \(u\leq s\),
and hence
\[
  2|\mathcal H|
  \leq u(u-1)
  \leq(s-1)u,
\]
as required.

Now let \(r\geq3\).  If
\(v\leq(s+1)(r-1)-1\), incidence counting gives
\[
  r|\mathcal H|
  \leq(v-r+1)|\partial\mathcal H|
  \leq\bigl((r-1)s-1\bigr)|\partial\mathcal H|.
\]
We may thus suppose that \(v\geq(s+1)(r-1)\).  Split at the last
vertex:
\begin{align*}
  \mathcal H_0
  &:=
  \{H\in\mathcal H:v\notin H\},\\
  \mathcal H_1
  &:=
  \{H\setminus\{v\}:H\in\mathcal H,\ v\in H\}.
\end{align*}
Clearly \(\nu(\partial\mathcal H_0)\leq s\).  We claim also that
\(\nu(\partial\mathcal H_1)\leq s\).  Otherwise choose pairwise
disjoint
\(T_1,\ldots,T_{s+1}\in\partial\mathcal H_1\).  Each
\(T_i\cup\{v\}\) lies in \(\partial\mathcal H\).  Moreover,
\[
  v-1-\left|\bigcup_{i=1}^{s+1}T_i\right|
  =
  v-1-(s+1)(r-2)
  \geq s.
\]
Choose distinct
\(x_1,\ldots,x_s\in[v-1]\setminus\bigcup_iT_i\).
Since \(\partial\mathcal H\) is shifted, the \(s+1\) sets
\[
  T_1\cup\{x_1\},\ldots,T_s\cup\{x_s\},
  T_{s+1}\cup\{v\}
\]
belong to \(\partial\mathcal H\) and are pairwise disjoint, a
contradiction.

Set \(q_r:=(r-1)s-1\).  The shadow contains the two families
\[
  \partial\mathcal H_0
  \qquad\text{and}\qquad
  \{T\cup\{v\}:T\in\partial\mathcal H_1\}.
\]
The first consists of sets avoiding $v$, while every set in the
second contains $v$; hence these parts are disjoint and give
\[
  |\partial\mathcal H|
  \geq
  |\partial\mathcal H_0|+|\partial\mathcal H_1|.
\]
The induction hypotheses yield
\[
  r|\mathcal H_0|\leq q_r|\partial\mathcal H_0|,
  \qquad
  (r-1)|\mathcal H_1|
  \leq q_{r-1}|\partial\mathcal H_1|.
\]
Finally,
\[
  (r-1)q_r-rq_{r-1}=s+1>0.
\]
Therefore
\[
\begin{split}
  r|\mathcal H|
  &\leq
  q_r|\partial\mathcal H_0|
  +\frac r{r-1}q_{r-1}|\partial\mathcal H_1|\\
  &\leq
  q_r\bigl(
    |\partial\mathcal H_0|+|\partial\mathcal H_1|
  \bigr)
  \leq q_r|\partial\mathcal H|,
\end{split}
\]
which completes the induction.
\end{proof}

Families
\(\mathcal G_1,\ldots,\mathcal G_{s+1}\) are
\emph{cross-dependent} if one cannot choose pairwise disjoint
\(G_i\in\mathcal G_i\) for all \(i\).

The following is the specialization $t=2s+1$ of Frankl's inequality for nested
cross-dependent families \cite[Theorem~3.1]{FranklMatching}.

\begin{lemma}[Frankl]
\label{lem:nested-cross-dependent}
Let $s\geq1$ and $\ell\geq1$ be integers, and let
$
\mathcal G_1,\ldots,\mathcal G_{s+1}
\subseteq\binom Y\ell
$
be cross-dependent families satisfying
\[
\mathcal G_1\supseteq\mathcal G_2
\supseteq\cdots\supseteq\mathcal G_{s+1},
\]
and suppose that
\(|Y|\geq(2s+1)\ell\).  Then
\begin{equation*}
  \sum_{i=1}^s|\mathcal G_i|
  +(s+1)|\mathcal G_{s+1}|
  \leq s\binom{|Y|}\ell.
\end{equation*}
\end{lemma}

\begin{proof}
Put $q=2s+1$, and choose $q$ pairwise disjoint $\ell$-sets
$B_1,\ldots,B_q$ uniformly at random in $Y$. Define
\[
 A_i:=\{j\in[q]:B_j\in\mathcal G_i\},\qquad a_i:=|A_i|.
\]
The sets $A_1\supseteq\cdots\supseteq A_{s+1}$ have no system
of distinct representatives, by cross-dependence. Hall's theorem
and nesting imply that $a_j\leq s+1-j$ for some $j$.
Indeed, a subcollection violating Hall's condition has a smallest
index $j$, union $A_j$, and at most $s+2-j$ members.
Set $h=s+1-j$. Since $a_i\leq q$ for $i<j$ and $a_i\leq h$
for $i\geq j$,
\[
 \sum_{i=1}^s a_i+(s+1)a_{s+1}
 \leq (s-h)q+(s+h+1)h
 =sq-h(s-h)\leq sq.
\]
For every $i$, symmetry gives
$\mathbb E a_i=q|\mathcal G_i|/\binom{|Y|}\ell$.
Taking expectations and cancelling $q$ proves the assertion.
\end{proof}

\subsection{An upper bound for \texorpdfstring{$\Phi_s(\cF)$}{Phi s(F)}}

We now prove the required bound for $\Phi_s(\cF)$.

\begin{lemma}\label{lem:matching-ce}
Let \(\cF\subseteq\binom{[n]}k\) satisfy \(\nu(\cF)\leq s\), where
\(k\geq2\) and \(n\geq(2s+1)k-s\).  Then
\begin{equation*}\tag{CE}\label{eq:matching-ce}
  \sum_{R\in\binom{[n]}{k-1}}
  \bigl(d_{\cF}(R)-s\bigr)_+
  \leq a_s(D-s).
\end{equation*}
%\stepcounter{equation}
Equality is attained by \(\cB_S\).
\end{lemma}

\begin{proof}
Lemma~\ref{lem:matching-shift} allows us to assume that \(\cF\) is
shifted.  Put
\[
  X:=[s+1],
  \qquad
  Y:=[s+2,n].
\]
For $P\subseteq X$ with $|P|\leq k$, define the
$(k-|P|)$-uniform {\em slice}
\[
  \cF(P)
  :=
  \left\{
    A\in\binom Y{k-|P|}:A\cup P\in\cF
  \right\}.
\]
If $r\geq1$, $\mathcal A\subseteq\binom Yr$, and $q\geq0$, abbreviate
\[
  \Phi_q(\mathcal A)
  :=
  \sum_{T\in\binom Y{r-1}}
  \bigl(d_{\mathcal A}(T)-q\bigr)_+.
\]

Consider a \((k-1)\)-set \(R=P\cup T\), where
\(P\subseteq X\), \(|P|=q\), and
\(T\in\binom Y{k-q-1}\).  Write
\[
  h:=d_{\cF(P)}(T),
  \qquad
  g:=
  \left|
    \{i\in X\setminus P:T\in\cF(P\cup\{i\})\}
  \right|.
\]
Then \(d_{\cF}(R)=h+g\).  If \(h>0\), there exists
\(y\in Y\setminus T\) such that
$
P\cup T\cup\{y\}\in\cF.
$
Since \(\cF\) is shifted, for every \(i\in X\setminus P\) we may
replace \(y\) by \(i\), obtaining
$
P\cup T\cup\{i\}\in\cF.
$
Thus \(T\in\cF(P\cup\{i\})\) for every \(i\in X\setminus P\), and
hence
\[
g=|X\setminus P|=s+1-q.
\]
If \(h=0\), only the inequality \(g\leq s+1-q\) is needed. When \(q=0\) and \(h=0\), we have
\[
(g-s)_+
=
\ind_{\{T\in\cF(\{s+1\})\}}.
\]
Indeed, if \(T\in\cF(\{s+1\})\), then, since \(\cF\) is shifted,
\(T\in\cF(\{i\})\) for every \(i\in X\), and hence \(g=s+1\).
If \(T\notin\cF(\{s+1\})\), then \(g\leq s\), so
\((g-s)_+=0\).
When \(q=0\) and \(h>0\), the preceding replacement argument gives
\(T\in\cF(\{s+1\})\) and \(g=s+1\), so the same expression
\(h+\ind_{\{T\in\cF(\{s+1\})\}}\) equals \(h+1\), as required.
It follows that the contribution of \(R\) to \(\Phi_s(\cF)\) is
\[
  \begin{cases}
    h+\ind_{\{T\in\cF(\{s+1\})\}},&q=0,\\
    h,&q=1,\\
    (h-q+1)_+,&q\geq2.
  \end{cases}
\]
Summing over all \(P,T\) gives
\begin{equation}\label{eq:matching-slice}
\begin{split}
  \Phi_s(\cF)
  ={}
  k|\cF(\emptyset)|+|\cF(\{s+1\})|+(k-1)\sum_{i=1}^{s+1}|\cF(\{i\})|+
  \sum_{q=2}^{\min\{s+1,k-1\}}
   \sum_{P\in\binom Xq}
  \Phi_{q-1}(\cF(P)).
\end{split}
\end{equation}

For every \(P\in\binom Xq\) with \(q\geq2\), we have
\(P\cap[s]\neq\emptyset\), because \(X\setminus[s]=\{s+1\}\).
Therefore the corresponding slice of \(\cB_{[s]}\) is the complete family
\(\binom Y{k-q}\).  Since \(\Phi_{q-1}\) is monotone under inclusion,
the entire \(q\geq2\) contribution in
\eqref{eq:matching-slice} is at most the corresponding contribution
for \(\cB_{[s]}\).  It remains to compare the combined \(q=0\) and
\(q=1\) contributions.

Set
\[
  \mathcal H:=\cF(\emptyset),
  \qquad
  \mathcal G_i:=\cF(\{i\})\quad(1\leq i\leq s+1),
  \qquad
  C:=\binom{|Y|}{k-1}.
\]
Since \(\cF\) is shifted, for \(1\leq i<j\leq s+1\) and
\(A\in\mathcal G_j\), replacing \(j\) by \(i\) in
\(A\cup\{j\}\in\cF\) gives \(A\cup\{i\}\in\cF\).  Hence
\(A\in\mathcal G_i\), and therefore
\[
\mathcal G_1\supseteq\mathcal G_2
\supseteq\cdots\supseteq\mathcal G_{s+1}.
\]
These nested families are also cross-dependent, since pairwise
disjoint choices \(G_i\in\mathcal G_i\), \(1\leq i\leq s+1\),
would give the \(s+1\) pairwise disjoint members
\[
G_i\cup\{i\}\in\cF,
\qquad 1\leq i\leq s+1,
\]
contrary to \(\nu(\cF)\leq s\). The threshold on \(n\) is exactly
\[
  |Y|=n-s-1\geq(2s+1)(k-1),
\]
so Lemma~\ref{lem:nested-cross-dependent} applies.

We also have
\[
  \partial\mathcal H\subseteq\mathcal G_{s+1}.
\]
Indeed, if \(T\in\partial\mathcal H\), then
\(T\cup\{y\}\in\mathcal H\) for some \(y\in Y\), and shifting \(y\)
to \(s+1\) gives \(T\in\mathcal G_{s+1}\). Moreover, \(\nu(\partial\mathcal H)\leq s\).  Otherwise, choose
pairwise disjoint
\(T_1,\ldots,T_{s+1}\in\partial\mathcal H\).  For each \(i\), choose
\(y_i\in Y\setminus T_i\) such that
\[
T_i\cup\{y_i\}\in\mathcal H\subseteq\cF.
\]
Since \(i<y_i\) and \(\cF\) is shifted, it follows that
\[
T_i\cup\{i\}\in\cF.
\]
These \(s+1\) members are pairwise disjoint, contradicting
\(\nu(\cF)\leq s\).  Lemma~\ref{lem:strong-shadow}
therefore yields
\[
  k|\mathcal H|
  \leq\bigl((k-1)s-1\bigr)|\partial\mathcal H|
  \leq\bigl((k-1)s-1\bigr)|\mathcal G_{s+1}|.
\]
Consequently,
\begin{align*}
  k|\mathcal H|+|\mathcal G_{s+1}|
   +(k-1)\sum_{i=1}^{s+1}|\mathcal G_i|&\leq
  (k-1)
  \left(
    \sum_{i=1}^s|\mathcal G_i|
    +(s+1)|\mathcal G_{s+1}|
  \right)\\
  &\leq(k-1)sC.
\end{align*}
For \(\cB_{[s]}\), the \(q=0\) contribution vanishes, while its
\(q=1\) slices satisfy
\[
\cB_{[s]}(\{i\})=\binom Y{k-1}
\quad\text{for }1\leq i\leq s,
\qquad
\cB_{[s]}(\{s+1\})=\emptyset.
\]
Hence the combined \(q=0\) and \(q=1\) contribution for
\(\cB_{[s]}\) is exactly
$
(k-1)sC.
$
Together with the comparison of the \(q\geq2\) contributions, this
gives
\[
\Phi_s(\cF)\leq\Phi_s(\cB_{[s]}).
\]
By the codegree distribution in \eqref{eq:matching-codegrees},
\(\Phi_s(\cB_{[s]})=a_s(D-s)\), proving
\eqref{eq:matching-ce}.
\end{proof}

The following consequence of Lemma~\ref{lem:matching-ce} gives
the sharp bound for $p=2$.

\begin{corollary}
\label{cor:matching-quadratic}
Under the assumptions of Lemma~\ref{lem:matching-ce},
\begin{equation*}
  \co_2(\cF)
  \leq
  sk|\cF|+Da_s(D-s)
  \leq
  a_sD^2+b_ss^2.
\end{equation*}
\end{corollary}

\begin{proof}
For \(0\leq d\leq D\),
\[
  d^2\leq sd+D(d-s)_+,
\]
because the difference on \(s\leq d\leq D\) is
\((D-d)(d-s)\), while \(d^2\leq sd\) for \(d\leq s\).
Summing and using \eqref{eq:matching-ce} gives the first inequality.
Theorem~\ref{thm:frankl-matching} and
\eqref{eq:matching-incidence-identity} give the second.  Equality
holds throughout for \(\cB_S\).
\end{proof}

\subsection{Transfer to every real exponent}

\begin{proof}[Proof of Theorem~\ref{thm:matching-main}]
Let \(d_R:=d_{\cF}(R)\).  Every \(d_R\) lies in \([0,D]\).
By Theorem~\ref{thm:frankl-matching} and
\eqref{eq:matching-incidence-identity},
\[
  \sum_Rd_R=k|\cF|
  \leq
  k\left(\binom nk-\binom{n-s}k\right)
  =a_sD+b_ss.
\]
The second hypothesis of Lemma~\ref{lem:hinge-transfer}, with
\(L=s\) and \(a=a_s\), is precisely \eqref{eq:matching-ce}.
That lemma therefore gives, for every real \(p\geq1\),
\[
  \co_p(\cF)
  =\sum_Rd_R^p
  \leq a_sD^p+b_ss^p
  =\co_p(\cB_S).
\]

For \(p=1\), equality is exactly equality in
Theorem~\ref{thm:frankl-matching}.  For \(p>1\), equality in
Lemma~\ref{lem:hinge-transfer} forces equality in the bound for
$\sum_R d_R$, and hence
equality in \eqref{eq:frankl-matching}.  In both cases Frankl's
equality statement gives \(\cF=\cB_S\) up to permutation.  Conversely, the
codegree distribution \eqref{eq:matching-codegrees} shows that
\(\cB_S\) attains equality for every \(p\geq1\).
\end{proof}

\section{Further work}
\label{sec:further-work}

This section considers two directions suggested by the preceding
results. We first apply the analytic inequalities to families of
subspaces of a finite vector space. We then formulate two problems
for families of subsets, specifying the combinatorial bounds still
needed to use these inequalities.

\subsection{Extensions to other incidence structures}
\label{subsec:other-incidence-structures}

Let $\mathcal X$ and $\mathcal Y$ be finite sets, and fix an
incidence relation between $\mathcal Y$ and $\mathcal X$, denoted by
$\leq$. For $\cF\subseteq\mathcal X$ and $Y\in\mathcal Y$, define
\[
  d_{\cF}(Y):=|\{X\in\cF:Y\leq X\}|.
\]
The analytic lemmas in Section~\ref{sec:two-moment} apply to
$(d_{\cF}(Y))_{Y\in\mathcal Y}$ whenever their stated assumptions
hold. These assumptions concern the upper bound for each degree and
bounds for the sums of degrees, squared degrees, or
$\sum_{Y\in\mathcal Y}(d_{\cF}(Y)-L)_+$. In particular,
Lemma~\ref{lem:endpoint-secant} applies when the degrees are
integers in $[0,D]$ and the proposed extremal family has degrees in
$\{0,1,D\}$.

Finite vector spaces give a concrete instance. Let $q$ be a prime
power, and let $V$ be an $n$-dimensional vector space over
$\mathbb F_q$. Write $[a]_q=(q^a-1)/(q-1)$ and
\[
  \qbinom ab:=\prod_{i=0}^{b-1}\frac{q^{a-i}-1}{q^{b-i}-1}
  \qquad(0\leq b\leq a),
\]
with the value $0$ when $b<0$ or $b>a$. The notation
$\qbinom Vk$ denotes the family of all $k$-dimensional subspaces
of $V$; its size is $\qbinom nk$. A family of subspaces is
\emph{intersecting} if any two members intersect in a subspace of
dimension at least $1$. A \emph{full star} is the family of all
$k$-dimensional subspaces containing a fixed subspace of dimension
$1$. Here $E\leq X$ means that $E$ is a subspace of $X$, so
$d_{\cF}(E)=|\{X\in\cF:E\leq X\}|$.
Assume from now on that $k\geq2$ and $n>2k$.
Bao, Chen, and Zhang \cite[Lemma~7]{BaoChenZhang} proved that,
for $\cF\subseteq\qbinom Vk$ and $m=|\cF|$,
\[
  \sum_{E\in\qbinom{V}{k-1}}d_{\cF}(E)^2
  \leq
  q[k-1]_q[n-k]_q m
  +\frac{[n]_q}{\qbinom{n}{k}}m^2.
\]
In Lemma~\ref{lem:two-moment-envelope}, take
\[
  N=\qbinom{n}{k-1},\quad
  M=\qbinom{n-1}{k-1},\quad
  L=1,\quad D=[n-k+1]_q,\quad
  a=\qbinom{n-1}{k-2}.
\]
The incidence identity gives $c=[k]_q$. In addition,
\[
 \frac{c}{M}=\frac{[n]_q}{\qbinom nk},\qquad
 D\left(c-\frac NM\right)=q[k-1]_q[n-k]_q.
\]
Thus the displayed quadratic bound is precisely $\mathcal B(m)$
in \eqref{eq:abstract-envelope-assumptions}.
Every degree is at most $\min\{m,D\}$, and $D\leq M$ follows
by fixing a $1$-dimensional subspace inside a $(k-1)$-dimensional
subspace. The other parameter conditions are immediate.
For an intersecting family, the classical
$q$-Erd\H{o}s--Ko--Rado theorem, with its equality statement
(see \cite[Theorem~3]{BaoChenZhang}), gives $m\leq M$.
Consequently, if
$k\geq2$, $n>2k$, and
$\cF\subseteq\qbinom{V}{k}$ is intersecting, then, for every real
$p\geq2$,
\begin{equation}\label{eq:grassmannian-consequence}
  \sum_{E\in\qbinom{V}{k-1}}d_{\cF}(E)^p
  \leq
  \qbinom{n-1}{k-2}[n-k+1]_q^p
  +q^{k-1}\qbinom{n-1}{k-1}.
\end{equation}
Equality holds if and only if $\cF$ is a full star. To see necessity,
for $p>2$ apply the equality statement in
Lemma~\ref{lem:two-moment-envelope} to obtain $m=M$.
For $p=2$, the displayed quadratic bound is a strictly increasing
function of $m\geq0$, so equality again forces $m=M$.
The equality statement in the $q$-Erd\H{o}s--Ko--Rado theorem
then gives a full star. Conversely,
the two positive degrees of $(k-1)$-dimensional subspaces in a full star
are $1$ and $[n-k+1]_q$, with respective multiplicities
$q^{k-1}\qbinom{n-1}{k-1}$ and
$\qbinom{n-1}{k-2}$; these values and multiplicities are the parameters used in
Lemma~\ref{lem:two-moment-envelope}.  This also verifies attainment in
\eqref{eq:grassmannian-consequence}.

For families of subspaces with pairwise intersection dimension at
least $t>1$, the analogue of Theorem~\ref{thm:codegree-main} would
follow from the two bounds in Lemma~\ref{lem:moment-input} with
sets replaced by subspaces. To obtain the analogue of
Theorem~\ref{thm:all-level-main} for degrees of $r$-dimensional
subspaces, one would need an analogue of Bey's inequality depending
on $|\cF|$, together with its equality conditions. For the matching
problem in a vector space, a matching consists of subspaces whose
sum is direct. Lemma~\ref{lem:hinge-transfer} would require an
analogue of \eqref{eq:matching-ce} for the degrees of subspaces.
Once these combinatorial bounds are available, the analytic
inequalities and their equality conditions apply as stated.

\subsection{Further problems for families of subsets}

We now consider two extensions for families of subsets.  The
first asks for a common
generalization of Theorems~\ref{thm:codegree-main} and
\ref{thm:all-level-main}, allowing both the intersection parameter
and the degree level to vary.

\begin{conjecture}\label{conj:general-t-r}
Let $1\leq t\leq k$, $1\leq r\leq k-1$, and let $p\geq2$ be real.
Suppose that
\[
  n\geq(t+1)(k-t+1).
\]
If $\cF\subseteq\binom{[n]}k$ is $t$-intersecting, then, for every
$T\in\binom{[n]}t$,
\[
  \ell_{r,p}(\cF)\leq\ell_{r,p}(\cS_T).
\]
\end{conjecture}

The range $n\geq(t+1)(k-t+1)$ in
Conjecture~\ref{conj:general-t-r} is the sharp range in Wilson's
theorem and the
Ahlswede--Khachatrian complete intersection theorem.  The case
$r=k-1$ is Theorem~\ref{thm:codegree-main}, while the case $t=1$ is
Theorem~\ref{thm:all-level-main}.

The restriction $p\geq2$ is essential.  Indeed, take
$k=2$, $t=r=1$, and $n=4$.  A full star has degree sequence
$(3,1,1,1)$, whereas the intersecting family
$\binom{[3]}2\subseteq\binom{[4]}2$ has degree sequence
$(2,2,2,0)$.  Consequently,
\[
  \ell_{1,p}(\cS_1)=3^p+3,
  \qquad
  \ell_{1,p}\left(\binom{[3]}2\right)=3\cdot2^p.
\]
For $1<p<2$, the function $x^{p-1}$ is strictly concave, and hence
\[
  2\cdot2^{p-1}>1+3^{p-1}.
\]
It follows that
\[
  3\cdot2^p>3^p+3,
\]
so even when $n=(t+1)(k-t+1)$, a full star need not maximize
the degree power sum for $1<p<2$.

The second direction is to extend the result for bounded matching number
from codegrees to arbitrary nontrivial degree levels.  Recall that
$\cB_S$ is the family of all $k$-sets meeting a fixed $s$-set $S$.

\begin{conjecture}\label{conj:matching-general-r}
Let $k\geq2$, $s\geq1$, $1\leq r\leq k-1$, and let $p\geq1$ be
real.  Suppose that
\[
  n\geq(2s+1)k-s.
\]
If $\cF\subseteq\binom{[n]}k$ satisfies $\nu(\cF)\leq s$, then
\[
  \ell_{r,p}(\cF)\leq\ell_{r,p}(\cB_S).
\]
Moreover, equality holds if and only if $\cF=\cB_S$ up to permutation.
\end{conjecture}

For $r=k-1$, Conjecture~\ref{conj:matching-general-r} is exactly
Theorem~\ref{thm:matching-main}.  Both conjectures are natural
targets for the method developed in Section~\ref{sec:two-moment},
but new sharp combinatorial input is required.  In
Conjecture~\ref{conj:general-t-r}, if $j=|R\cap T|$, then
\[
  d_{\cS_T}(R)
  =
  \binom{n-r-t+j}{k-r-t+j}.
\]
Thus, as $j$ varies, the $r$-degrees of a full $t$-star can take
several positive values. The first application uses only the values
$0,1,D$, and the application with $t=1$ uses two positive values.
To treat general $t,r$, one would need an upper bound for $x^p$
that agrees with $x^p$ at all the displayed degree values, together
with sharp bounds for the required sums of powers in terms of
$|\cF|$.

For the matching conjecture, the degrees of $\cB_S$ still take
two positive values:
\[
  d_{\cB_S}(R)
  =
  \begin{cases}
    \displaystyle\binom{n-r}{k-r},&R\cap S\neq\emptyset,\\[6pt]
    \displaystyle\binom{n-r}{k-r}
    -\binom{n-s-r}{k-r},&R\cap S=\emptyset.
  \end{cases}
\]
Thus Lemma~\ref{lem:hinge-transfer} gives the required analytic
inequality. Put
\[
  L:=\binom{n-r}{k-r}-\binom{n-s-r}{k-r}.
\]
The missing bound is
\[
  \sum_{R\in\binom{[n]}r}(d_{\cF}(R)-L)_+
  \leq
  \sum_{R\in\binom{[n]}r}(d_{\cB_S}(R)-L)_+.
\]
The proof of Lemma~\ref{lem:matching-ce} uses
$\cF(\emptyset)$, its shadow $\partial\cF(\emptyset)$, and the
nested families $\cF(\{i\})\subseteq\binom Y{k-1}$. It does not
directly extend to general $r$. Proving the displayed inequality
would give a route to Conjecture~\ref{conj:matching-general-r}.

For completeness, a \emph{$\Delta$-system} is a family whose
distinct members all have the same pairwise intersection. For
$\lambda>1$, a nonempty family $\mathcal H$ is
\emph{$\lambda$-spread} if
\[
  |\{H\in\mathcal H:A\subseteq H\}|
  \leq\lambda^{-|A|}|\mathcal H|
  \quad\text{for every nonempty }A\subseteq\bigcup_{H\in\mathcal H}H.
\]
It would be useful to investigate whether the $\Delta$-system
method or the method of extracting spread families can prove these
conjectures first for sufficiently large $n$. Reaching the stated
thresholds would require additional quantitative estimates.

\bigskip
\noindent{\bf Declaration of AI usage.} AI tools were used to assist with preliminary exploration and manuscript revision. The authors are responsible for the mathematical content and the final version of the manuscript.

\end{document}